\documentclass[reqno,11pt]{amsart}

\usepackage{amsmath,amsfonts,amssymb,amsthm,epsfig,stix}

\def\R{\mathbb R}
\def\N{\mathbb N}
\def\D{D}
\def\div{\operatorname{div}}
\def\dx{\ dx}

\def\dt{\ dt}

\numberwithin{equation}{section}
\newtheorem{theorem}{Theorem}[section]

\newtheorem{corollary}[theorem]{Corollary}

\newtheorem{lemma}[theorem]{Lemma}

\theoremstyle{definition}
\begingroup

\endgroup

\begin{document}

\title[Electrostatic Born--Infeld equations]{The electrostatic Born--Infeld equations with integrable charge densities}
\author{A. Haarala}
\address[A. Haarala]{Matematiikan ja tilastotieteen laitos, Helsingin yliopisto, Helsinki, Finland}
\email[Akseli Haarala]{akseli.haarala@helsinki.fi}

\begin{abstract}
We study the minimizer of the electrostatic Born--Infeld energy
\begin{equation*}
\int_{\R^n}1-\sqrt{1-|\D v|^2}\dx-\int_{\R^n}\rho v\dx,
\end{equation*}
which vanishes at infinity. We show that the minimizer $u$ is strictly spacelike and it is a weak solution to
\begin{equation*}
-\div\Big(\frac{\D u}{\sqrt{1-|\D u|^2}}\Big)=\rho,
\end{equation*}
provided that $\rho$ is in the dual space of the solution space and $\rho\in L^p(\R^n)$, for some $p>n\geq 3$. Moreover, we have $u\in C^{1,\alpha}(\R^n)$ for some $\alpha\in (0,1)$.
\end{abstract}
\maketitle

{\small


\smallskip

}
\bigskip


\section{Introduction}\label{section:intro}
The Lorentz--Minkowski space $\mathbb{L}^{n+1}$, $n\geq 2$, studied in the theory of special relativity, is the space $\R^{n+1}$ equipped with the non-degenerate bilinear form $\langle\cdot,\cdot\rangle_{\mathbb{L}^{n+1}}$ defined by
\begin{equation*}
\langle (x,t),(y,s)\rangle_{\mathbb{L}^{n+1}}:=\langle x,y\rangle_{\R^{n}}-ts, \quad\text{for all } t,s\in\R \text{ and } x,y\in\R^n,
\end{equation*}
where $\langle \cdot,\cdot\rangle_{\R^{n}}$ is the usual Euclidean inner product in $\R^n$. Let $\Omega\subset\R^n$ be a domain and let $u:\Omega\rightarrow\R$ be a continuously differentiable function. The graph $M:=\{(x,u(x))\in \mathbb{L}^{n+1}:x\in\Omega\}$ of $u$ is a hypersurface in $\mathbb{L}^{n+1}$. The tangent space of $M$ at $p\in M$ is denoted by $T_pM$ and can be identified with the subspace  
\begin{equation*}
T_pM=\big\{\big(x,\langle\D u(p),x\rangle\big)\in \mathbb{L}^{n+1}:x\in\R^n\big\}.
\end{equation*}
The bilinear form $\langle\cdot,\cdot\rangle_{\mathbb{L}^{n+1}}$ induces a bilinear form $\langle\cdot,\cdot\rangle_M$ on the tangent space $T_pM$ through the inclusion $T_pM\subset \mathbb{L}^{n+1}$, namely
\begin{equation*}
\langle X,Y\rangle_M:=\langle X,Y\rangle_{\mathbb{L}^{n+1}}, \quad\text{for all }X,Y\in T_pM.
\end{equation*}
The induced bilinear form defines a Riemannian metric on $M$ if and only if $|\D u|<1$ in $\Omega$. We use the following terminology introduced in \cite{BS}.
\begin{enumerate}
\item[i)] A function $u\in C^1(\Omega)$ is strictly spacelike if $|\D u|<1$ in $\R^n$;
\item[ii)] A function $u:\Omega\rightarrow\R$ is spacelike if $|u(x)-u(y)|<|x-y|$, for all points $x,y\in \Omega$ such that $x\not=y$ and that the line segment $[x,y]$ connecting $x$ and $y$ is contained in $\Omega$;
\item[iii)] A function $u\in W^{1,\infty}_{loc}(\Omega)$ is weakly spacelike if $|\D u|\leq 1$ a.e. in $\Omega$.
\end{enumerate}

Since the graph of any strictly spacelike function equipped with the induced metric is a Riemannian hypersurface, its area with respect to the Riemannian metric is well-defined and is given by
\begin{equation}\label{eq:area}
\int_\Omega \sqrt{1-|\D u|^2}\dx.
\end{equation}
Suppose $\Omega\subset\R^n$ is a bounded domain. One is immediately led to the natural problem of finding the area maximizing strictly spacelike hypersurface with given boundary data $\varphi:\partial\Omega\rightarrow\R$. This problem is analogous to the classical Plateau problem in the Euclidean space. Note that the area functional (\ref{eq:area}) is well-defined even for weakly spacelike hypersurfaces. When considering variational problems involving the area functional (\ref{eq:area}), it is convenient to work with the family of weakly spacelike functions due to its obvious compactness properties. Indeed, it is not difficult to prove that there exists a unique function $u:\Omega\rightarrow\R$ maximizing the area functional (\ref{eq:area}) amongst the admissible class
\begin{equation}\label{eq:admissible}
\mathcal{A}:=\big\{v\in W^{1,\infty}_{loc}(\Omega): v \text{ weakly spacelike and }v=\varphi \text{ on }\partial\Omega \big\},
\end{equation}
if it is non-empty. The boundary values should be understood in the following sense: we say that $v=\varphi$ on $\partial\Omega$ if
\begin{equation*}
\lim_{\substack{x\rightarrow x_0 \\ x\in L}}v(x)=\varphi(x_0),
\end{equation*}
for any $x_0\in\partial\Omega$ and any open line segment $L\subset\Omega$ with endpoint $x_0$. The maximizer $u$ is \textit{a priori} only weakly spacelike but it was proved in \cite{BS} that $u$ is strictly spacelike provided that there exists at least one spacelike function $v$ such that $v=\varphi$ on $\partial\Omega$. Moreover, the maximizer $u$ is a weak solution of the maximal surface equation
\begin{equation}\label{eq:maximal}
\div\Big(\frac{\D u}{\sqrt{1-|\D u|^2}}\Big)=0.
\end{equation}
The left hand side of (\ref{eq:maximal}) is the Lorentz mean curvature of the graph of $u$. Hypersurfaces with zero Lorentz mean curvature are called maximal surfaces. It was proved in \cite{C} ($n\leq 4$) and \cite{CY} ($n>4$) that maximal surfaces have the Bernstein property, that is, entire maximal surfaces are affine.

Let us consider the problem of prescribed Lorentz mean curvature. Let $H:\Omega\rightarrow\R$ be a bounded and measurable function. Define the variational integral
\begin{equation}\label{eq:prescribedcurvatureintegral}
I_H(u):=\int_\Omega \sqrt{1-|\D u|^2}\dx-\int_\Omega H u\dx.
\end{equation}
By the direct method in calculus of variations, there exists a unique function $u\in\mathcal{A}$ that maximizes (\ref{eq:prescribedcurvatureintegral}) amongst the family $\mathcal{A}$ (see \cite[Section 1]{BS} for details). It is not clear that the maximizer should be strictly spacelike or satisfy any differential equation. In \cite{BS} Bartnik and Simon studied the Lorentz mean curvature equation and proved in particular the following result which gives a somewhat complete picture of the regularity problem in the case of bounded mean curvature.

\begin{theorem}{(\cite[Corollary 4.2.]{BS})}\label{th:BS}
Let $\Omega\subset\R^n$ be a bounded domain and let $H:\Omega\rightarrow\R$ be a bounded and measurable function. Suppose that $\varphi:\partial\Omega\rightarrow\R$ has a weakly spacelike extension to $\Omega$. The unique maximizer $u$ of (\ref{eq:prescribedcurvatureintegral}) among the class (\ref{eq:admissible}) is strictly spacelike in $\Omega\setminus K$, where
\begin{equation*}
K:=\bigcup\big\{[x,y]:x,y\in\partial\Omega,\ (x,y)\subset\Omega,\ |\varphi(x)-\varphi(y)|=|x-y|\big\}.
\end{equation*}
Moreover, the maximizer $u$ satisfies the prescribed Lorentz mean curvature equation
\begin{equation}\label{eq:prescribedcurvature}
\div\Big(\frac{\D u}{\sqrt{1-|\D u|^2}}\Big)=H, \quad\text{in }\Omega\setminus K,
\end{equation}
in the weak sense.
\end{theorem}

We will briefly describe the proof of Theorem \ref{th:BS}. Let us start by explaining the structure of the singular set $K$. One essential tool in proving the regularity of the maximizer is the "anti-peeling" lemma, see \cite[Theorem 3.2]{BS}. Suppose there exists a line segment $[x,y]\subset\Omega$ such that the maximizer $u$ is lightlike on $[x,y]$, that is
\begin{equation*}
u\big(x+t(y-x)\big)=u(x)+t|y-x|,
\end{equation*}
for all $t\in (0,1)$. Then $[x,y]$ can be extended into a line segment $[\tilde{x},\tilde{y}]$ between two boundary points $\tilde{x},\tilde{y}\in\partial\Omega$ such that $u$ is lightlike on $[\tilde{x},\tilde{y}]$. It follows immediately from the "anti-peeling" lemma that the maximizer is spacelike in $\Omega\setminus K$. The proof of the "anti-peeling" lemma is based on a barrier construction using the comparison principle (\cite[Lemma 1.2]{BS}) and the barrier functions given by
\begin{equation}\label{eq:barriers}
w_{a,\Lambda}(x)=\int_0^{|x|}\frac{a-\frac{\Lambda}{n}t^n}{\sqrt{t^{2(n-1)}+\big(a-\frac{\Lambda}{n}t^n\big)^2}}\dt,
\end{equation}
for $a\geq 0$ and $\Lambda\geq 0$. The radially symmetric function $w_{a,\Lambda}$ is strictly spacelike in $\R^n\setminus\{0\}$ and has a conelike singularity at the origin. It satisfies
\begin{equation*}
\div\Big(\frac{\D w_{a,\Lambda}}{\sqrt{1-|\D w_{a,\Lambda}|^2}}\Big)=-\Lambda+n\omega_n a\delta_0,
\end{equation*}
where $\omega_n$ is the volume of the unit ball in $\R^n$ and $\delta_0$ is the Dirac mass at the origin. Moreover, for fixed $\Lambda\geq 0$ the functions $w_{a,\Lambda}$ converge uniformly on compact sets to the cone $x\mapsto |x|$ as $a\rightarrow \infty$. 

One proves that for smooth mean curvature data $H$ and smooth, strictly spacelike boundary data $\varphi$ there exists a smooth strictly spacelike solution $u$ of the Dirichlet problem
\begin{equation}\label{300}
\begin{cases}
\div\Big(\frac{\D u}{\sqrt{1-|\D u|^2}}\Big)=H &\quad \text{in }\Omega;\\
u=\varphi & \quad \text{on }\partial \Omega. 
\end{cases}
\end{equation}
The proof is based on \textit{a priori} gradient estimate and Leray--Schauder fixed point theorem. For the proof of the \textit{a priori} gradient estimate, consider the quantity
\begin{equation}\label{eq:nu}
\nu=\frac{1}{\sqrt{1-|\D u|^2}}.
\end{equation}
Using the barrier functions defined in (\ref{eq:barriers}) it is possible to prove a boundary estimate for the quantity $\nu$ (see \cite[Corollary 3.4]{BS}). Note that the construction depends on the boundary data $\varphi$.

The boundary estimate can then be extended to a global estimate. Indeed, by linearizing the equation (\ref{300}) it can be proven that $\nu$ satisfies
\begin{equation}\label{eq:nueq}
\div\big(A\D\nu\big)\geq \langle \D H,\D u\rangle,
\end{equation}
for a certain coefficient matrix function $A$. The ratio of the biggest and the smallest eigenvalue of $A$ is controlled by $\nu^2$. Using the equation (\ref{eq:nueq}), it is possible to derive a Cacciappoli type inequality for $\nu$. The Cacciappoli type inequality is combined with Sobolev inequality to prove self-improvement estimate which can be iterated and combined with an $L^2$-estimate to derive a global bound for $\nu$ (see \cite[Theorem 3.5]{BS}).

Once the \textit{a priori} estimate on the gradient of the solution is established, the equation (\ref{300}) essentially reduces to a uniformly elliptic equation. The regularity theory of uniformly elliptic equations is well-known. A fixed point argument using Leray-Schauder fixed point theorem proves the existence of a strictly spacelike smooth solution $u$ of the Dirichlet problem (\ref{300}) with smooth data (see \cite[Theorem 3.6]{BS}).

The smooth and strictly spacelike solutions of (\ref{300}) also satisfy the following \textit{a priori} derivative estimate (see \cite[Lemma 2.1]{BS}). Let $x_0\in \Omega$. Then there exists $\gamma=\gamma(n)\in(0,\frac{1}{n})$ and $C=C(n)>0$ such that we have 
\begin{equation}\label{eq:BSderivativeestimate}
R^{-n}\int_{K_R(x_0)}\nu^{-(\gamma+1)}\dx+R^{2-n}\int_{K_R(x_0)}|\D^2 u|^2\dx\leq C\exp\big(1+R^2\Vert H\Vert_\infty\big)\nu(x_0)^{-\gamma},
\end{equation}
for all $R>0$, such that $K_{2R}(x_0)\subset\subset \Omega$, where
\begin{equation}
K_R(x_0):=\big\{x\in\Omega: [x,x_0]\subset\Omega, (|x-x_0|^2-|u(x)-u(x_0)|^2)^\frac{1}{2}\leq R\big\}.
\end{equation}
It is clear that $B_R(x_0)\subset K_R(x_0)$ if $K_R(x_0)\subset\subset \Omega$. Note that $K_R(x_0)$ does not depend on the infinitesimal properties of $u$ but rather on how $u$ behaves on fixed scales. 

Another important consequence of the "anti-peeling" lemma is that the uniformly converging sequence $(u_k)_{k=1}^\infty$ of spacelike functions, with uniformly bounded Lorentz mean curvature, whose limit is spacelike, is "uniformly coarsely spacelike". Namely, for any $r>0$ there exists $\theta>0$ such that
\begin{equation*}
\sup_{\substack{x,y\in\Omega \\ |x-y|>r}}\frac{|u_k(x)-u_k(y)|}{|x-y|}\leq 1-\theta,
\end{equation*}
for all $k\in\N$ (see \cite[Theorem 3.3]{BS}). 

By a standard approximation argument the maximizer $u$ can be approximated locally in $\Omega\setminus K$ by smooth strictly spacelike functions $u_k$, $k\in\N$, with uniformly bounded Lorentz mean curvature. Since the maximizer is spacelike the approximating functions are "uniformly coarsely spacelike". Thus the estimate (\ref{eq:BSderivativeestimate}) holds for all approximating functions $u_k$, $k\in\N$, in some small neighbourhood of a point $x_0\in\Omega$ with some $R_0>0$ independent of $k$. It follows from (\ref{eq:BSderivativeestimate}) that if $|\D u_k|$ is close to $1$ at a point then the same holds in a small neighbourhood of the point. The proof proceeds to show that the $|\D u_k|$ is locally uniformly bounded away from $1$. Indeed, if that is not the case, then the derivative estimate (\ref{eq:BSderivativeestimate}) shows that the maximizer $u$ must be affine with slope $1$ in some small ball. This is not possible, since it was already concluded that $u$ is spacelike in $\Omega\setminus K$. This provides a derivative bound for the minimizer $u$. The regularity of $u$ follows immediately by standard methods for uniformly elliptic quasilinear equations.

\bigskip
\noindent
The prescribed mean curvature equation discussed above appears also in a different context in physics. In the rest of the paper we assume $n\geq 3$. The classical Maxwell's equations lead to the unsatisfying conclusion that the energy of the electro-magnetic field generated by a point charge is infinite. To remedy this problem  Born (\cite{B1},\cite{B2}) and Born and Infeld (\cite{BI1},\cite{BI2}) proposed a new nonlinear electrodynamical model. In the classical Maxwell's theory the electromagnetic Lagrangian is given by
\begin{equation*}
L_M:=\frac{1}{2}\big(|E|^2-|B|^2\big),
\end{equation*}
where $B:\mathbb{L}^{3+1}\rightarrow\R^3$ is the magnetic field and $E:\mathbb{L}^{3+1}\rightarrow\R^3$ is the electric field. Born and Infeld proposed that $L_M$ be replaced by the Lagrangian
\begin{equation*}
L_{BI}:=b^2\Big(1-\sqrt{1+\frac{1}{b^2}\big(|B|^2-|E|^2\big)-\frac{1}{b^4}\langle B,E\rangle^2}\Big),
\end{equation*}
where $b>0$ is the absolute field constant. Let us consider the electrostatic case, that is, assume $B\equiv 0$ and $E$ independent of the time variable. By Faraday's law, there exists an electrostatic potential $\phi$ such that $E=\D \phi$, and the Lagrangian $L_{BI}$ reduces to $b^2(1-\sqrt{1-|\D \phi|^2/b^2})$. Given a charge distribution $\rho$ the electrostatic potential $\phi$ therefore minimizes the energy
\begin{equation*}
\int_{\R^n}b^2\Big(1-\sqrt{1-\frac{|\D \phi|^2}{b^2}}\Big)dx-\int_{\R^n}\rho\phi,
\end{equation*}
which is essentially the same as the variational integral (\ref{eq:prescribedcurvatureintegral}) introduced in the prescribed Lorentz mean curvature problem. We refer to \cite{BI2} and the introduction of \cite{BdAP} for more thorough discussion on the physics of the problem.

We follow the mathematical formulation of the problem given in \cite{BdAP}. Define the family
\begin{equation*}
\mathcal{X}:=\big\{v\in D^{1,2}_0(\R^n):|\D v|\leq 1 \text{ a.e. in }\R^n\big\},
\end{equation*}
where $D^{1,2}_0(\R^n)$ is the completion of $C_c^\infty(\Omega)$ with respect to the norm
\begin{equation*}
\Vert u\Vert_{D^{1,2}_0(\Omega)}:=\Big(\int_{\R^n}|\D u|^2\dx\Big)^\frac{1}{2}.
\end{equation*}
Let $\rho\in \mathcal{X}^*$ be a given charge distribution in the dual of $\mathcal{X}$. There exists a unique $u\in\mathcal{X}$ that minimizes
\begin{equation}\label{eq:variationalintegral}
I_\rho(v):=\int_{\R^n}1-\sqrt{1-|\D v|^2}dx-\langle\rho,v\rangle
\end{equation}
amongst $\mathcal{X}$ (see \cite[Proposition 2.3]{BdAP}). If the minimizer $u$ is strictly spacelike, it is easy to prove that $u$ satisfies the equation
\begin{equation}\label{eq:mainequation}
-\div\Big(\frac{\D u}{\sqrt{1-|\D u|^2}}\Big)=\rho \quad\text{ in }\R^n,
\end{equation}
in the weak sense. However, it is not \textit{a priori} clear that $u$ should be strictly spacelike or satisfy (\ref{eq:mainequation}). Any weak solution $u\in\mathcal{X}$ of (\ref{eq:mainequation}) is also a minimizer of (\ref{eq:variationalintegral}).

Actually, little is known about the regularity of the minimizer in general. By Theorem \ref{th:BS} the minimizer $u$ is stricly spacelike if $\rho\in \mathcal{X}^*\cap L^\infty_{loc}(\R^n)$ (see \cite[Theorem 1.5]{BdAP} for details). It is also known that the minimizer is a weak solution of (\ref{eq:mainequation}) if $\rho$ is radially distributed (see \cite[Theorem 1.4]{BdAP}). A natural question, given the physical origin of the problem, is what happens when $\rho$ is a finite linear combination of Dirac masses, i.e.  $\rho=\sum_{k=1}^N a_k\delta_{x_k}$ for some $N\in\N$, $a_k\in\R$, $x_k\in\R^n$, $k=1,\dots,N$. This corresponds to the electrostatic potential generated by finitely many point charges. It has been proven (see \cite[Theorem 1.6]{BdAP}) that the minimizer is strictly spacelike in $\R^n\setminus\{x_1,\dots,x_N\}$ and satisfies (\ref{eq:mainequation}), if the charges $a_k$ are small enough or the points $x_k$ are far enough from each other. See \cite{BdAP} for further discussion. In the case of a single (positive) point charge the solution is given by the function $w_{a,0}$, defined in (\ref{eq:barriers}), where $a$ corresponds to the magnitude of the charge. Hence, the electrostatic potential generated by a single point charge at origin is strictly spacelike in $\R^n\setminus\{0\}$ and has a singularity at the origin. The regularity of the minimizer in the general case of finitely many point charges is an interesting open problem.

The following result was proven in \cite{BI}, see (\cite[Theorem 1.5]{BI}) and (\cite[Theorem 1.6]{BI}).

\begin{theorem}\label{th:BI}
Let $\rho\in L^{q}(\R^n)\cap L^{m}(\R^n)$, for some $q>2n$ and $m\in[1,2n/(n+2)]$. Let $u\in\mathcal{X}$ be the unique minimizer of (\ref{eq:variationalintegral}) in $\mathcal{X}$. Then $u\in W^{2,2}_{loc}(\R^n)$.

Moreover, there exists $c=c(n,q,m)>0$ such that if $\Vert \rho\Vert_q+\Vert\rho\Vert_m<c$ then $u$ is strictly spacelike, $u$ is a weak solution of the equation (\ref{eq:mainequation}) and $u\in C^{1,\alpha}_{loc}(\R^n)$ for some $\alpha=\alpha(n,m,q)\in(0,1)$.
\end{theorem}

In \cite{BI} the authors conjectured that the minimizer $u\in C_{loc}^{1,\alpha}(\R^n)$ for some $\alpha\in (0,1)$ if $\rho\in\mathcal{X}^*\cap L_{loc}^p(\R^n)$ for some $p>n$. Our main result proves the conjecture in the case $\rho\in \mathcal{X}^*\cap L^p(\R^n)$, $p>n$.

\begin{theorem}\label{th:main}
Let $\rho\in \mathcal{X}^*\cap L^p(\R^n)$ for some $p>n$. Let $u\in \mathcal{X}$ be the unique minimizer of (\ref{eq:variationalintegral}) in $\mathcal{X}$. Then there exists $\theta=\theta(n,p,\Vert\rho\Vert_{\mathcal{X}^*},\Vert\rho\Vert_p)>0$ such that
\begin{equation*}
\Vert\D u\Vert_\infty\leq 1-\theta.
\end{equation*}
Also, there exists $\alpha=\alpha(n,p,\theta)>0$ such that $u\in C^{1,\alpha}(\R^n)$ and 
\begin{equation*}
\Vert u\Vert_{C^{1,\alpha}(\R^n)}\leq C.
\end{equation*}
where $C=C(n,p,\theta,\Vert\mu\Vert_{\mathcal{X}^*},\Vert\mu\Vert_p)>0$.
\end{theorem}

The main derivative estimate in \cite{BI} (see \cite[Proposition 3.1]{BI}) is based on the proof of the derivative estimate (\ref{eq:BSderivativeestimate}). The essential difference between the derivative estimate in \cite{BI} and (\ref{eq:BSderivativeestimate}) is an error term depending on the $L^q$ norm of $\rho$. To prove that the minimizer is strictly spacelike the the error term is required to be small. The assumption on the smallness of the $L^m$ and $L^q$ norms of $\rho$ in Theorem \ref{th:BI} comes from this requirement.

The proof of Theorem \ref{th:main} is based on \textit{a priori} derivative bound and Leray--Schauder fixed point theorem. In Section \ref{section:Apriori} we prove \textit{a priori} bound for the derivatives of smooth solutions of (\ref{eq:mainequation}) in terms of the $L^p$ norm of $\rho$ (see Theorem \ref{th:derivativeapriori}). The proof of our \textit{a priori} derivative bound is based on the ideas of \cite[Theorem 3.5]{BS} and can be divided into two parts. First, we prove that the supremum of the quantity $\nu$, defined in (\ref{eq:nu}), can be controlled locally by the $L^p$ norms of $\nu$ and $\rho$ (see Theorem \ref{th:nuestimate}). The proof is based on a Cacciappoli type inequality derived from the linearized equation. Second, we prove that the $L^p$ norm of $\nu$ can be controlled globally by the $L^p$ norm of $\rho$ (see Theorem \ref{th:integrability}). 

In Section \ref{section:proof} we use Schauder fixed point theorem to prove the existence of a solution. The use of Schauder fixed point theorem is a standard method in the theory of quasilinear equations (see \cite[Chapter 11]{GT}) and has been applied in the study of the Lorentz mean curvature operator in \cite[Theorem 3.6]{BS}. We will make use of the classical Schauder estimates, the Calderon--Zygmund estimates for the second derivatives of uniformly elliptic non-divergence form equations and some simple decay estimates for the solutions. The use of the Calderon-Zygmund estimates and decay estimates is required to deal with the unbounded domain. See Section \ref{section:Preliminaries} for precise form of these classical results we are using.

\subsubsection*{Acknowledgements} A. Haarala is financially supported by the Academy of Finland, project \#308759. The author would like to thank Denis Bonheure and Alessandro Iacopetti for reading the paper and giving many valuable comments.


\section{Preliminaries and notation}\label{section:Preliminaries}

\subsection*{Notation}
It is convenient to introduce the following notation for the electrostatic Lagrangian. We define
\begin{equation*}
F(\xi):=1-\sqrt{1-|\xi|^2},
\end{equation*}
for all $\xi\in\R^n$ such that $|\xi|\leq 1$. The function $F$ is smooth and strictly convex in the open unit ball $B(0,1)$.

Next we introduce some functional spaces. Let $\Omega\subset \R^n$ be a domain, possibly unbounded, and let $k\in \N$. For any $1\leq p\leq \infty$ the space $L^p(\Omega)$ is the usual space of $p$-integrable functions with the standard norm denoted $\Vert\cdot\Vert_{p,\Omega}$. If $\Omega=\R^n$ we drop it from the notation.

We denote by $C^k(\Omega)$ the space of $k$-times continuously differentiable functions in $\Omega$. We say $f\in C^k(\bar{\Omega})$ if $f$ and its derivatives up to order $k$ admit continuous extensions to $\bar{\Omega}$. Furthermore, we say $f\in C_b^k(\Omega)$ if $f\in C^k(\Omega)$ and
\begin{equation*}
\Vert f\Vert_{C_b^k(\Omega)}:=\sum_{j=0}^k\big\Vert\big(\sum_{|\beta|=j}|\D^\beta f|^2\big)^{1/2}\big\Vert_{\infty,\Omega}<\infty,
\end{equation*}
where in the sum $\beta$ denotes a multi-index.

Let $\alpha\in (0,1]$. We say $f\in C^{0,\alpha}(\Omega)$ if $f$ is bounded and
\begin{equation*}
[f]_{\alpha,\Omega}:=\sup_{x,y\in\Omega}\frac{|f(x)-f(y)|}{|x-y|^\alpha}<\infty.
\end{equation*}
We say $f\in C^{k,\alpha}(\Omega)$ if $f\in C_b^k(\Omega)$ and $\D^\beta f\in C^{0,\alpha}(\Omega)$ for all multi-indeces $\beta$ such that $|\beta|=k$. We define the norm
\begin{equation*}
\Vert f\Vert_{C^{k,\alpha}(\Omega)}:=\sum_{j=0}^k\big\Vert\big(\sum_{|\beta|=j}|\D^\beta f|^2\big)^{1/2}\big\Vert_{\infty,\Omega} + \sum_{|\beta|=k}[\D^\beta f]_{\alpha,\Omega}.
\end{equation*}

We denote by $W_{loc}^{k,1}(\Omega)$ the Sobolev space of $k$ times weakly differentiable functions in $\Omega$. Let $1\leq p\leq \infty$. In the usual way, we define the space
\begin{equation*}
W^{k,p}(\Omega):=\{u\in W_{loc}^{k,1}(\Omega):\D^\beta u\in L^p(\Omega)\text{ for all multi-indeces }\beta \text{ s.t. }|\beta|\leq k\},
\end{equation*}
equipped with the norm
\begin{equation*}
\Vert u\Vert_{W^{k,p}(\Omega)}=\Big(\sum_{|\beta|\leq k}\int_{\Omega}|\D^\beta u|^p\dx\Big)^\frac{1}{p}.
\end{equation*}

\subsection*{Uniformly elliptic divergence form equations}
We recall the standard H\"older estimates for the weak solutions of uniformly elliptic divergence form equations with measurable coefficients.

\begin{theorem}{(\cite[Theorem 8.24.]{GT})}\label{th:C1,alpha}
Let $x\in\R^n$ and $R>0$. Denote $B_R:=B(x,R)$ and $B_{R/2}:=B(x,R/2)$. Let $A:B_R\rightarrow\R^{n\times n}$ be measurable and satisfy $A^T=A$ and $\lambda I\leq A\leq \Lambda I$, for some $\lambda,\Lambda>0$. Let $g\in L^{p/2}(B_R)$ and let $f\in L^p(B_R;\R^n)$ for some $p>n$. Suppose $v\in W^{1,2}(B_R)$ is a weak solution of the equation
\begin{equation*}
\div(A\D v)=g+\div(f),\quad\text{in }B_R,
\end{equation*}
then
\begin{equation*}
[v]_{\alpha,B_{R/2}}\leq C\big(\Vert v\Vert_{\infty,B_R}+\Vert g\Vert_{p/2,B_R}+\Vert f\Vert_{p,B_R}\big),
\end{equation*}
where $C=C(n,\Lambda/\lambda,p,R)>0$ and $\alpha=\alpha(n,p,\Lambda/\lambda)\in(0,1]$.
\end{theorem}

\subsection*{Uniformly elliptic non-divergence form equations}
We recall some classical results for uniformly elliptic non-divergence form equations. We start with the Calderon--Zygmund estimates followed by Schauder estimates and then note some elementary decay estimates. 

We follow \cite[Chapter 9]{GT} where Calderon--Zygmund estimates are proven for uniformly elliptic non-divergence form equations with uniformly continuous coefficients on bounded domains. For the sake of completeness we present the trivial modifications that allow us to deal with the unbounded domain $\R^n$. See Theorem \ref{th:existencenondivergenceintro} for the result we will use later. Consider the non-divergence form equation
\begin{equation}\label{eq:nondivergenceintro}
\langle A,\D^2 u\rangle = \mu \quad\text{a.e. in }\R^n,
\end{equation}
where $A:\R^n\rightarrow\R^{n\times n}$ and $\mu:\R^n\rightarrow \R$ are measurable and $u\in W_{loc}^{2,1}(\R^n)$. The coefficients $A:\R^n\rightarrow\R^{n\times n}$ are matrix valued functions satisfying the following two assumptions
\begin{equation}\label{eq:coef}
\begin{cases}
\big(A(x)\big)^T=A(x), &\quad\text{for a.e. }x\in\R^n,\\
\lambda I\leq A(x)\leq \Lambda I, &\quad\text{for a.e. }x\in\R^n. \\
\end{cases}
\end{equation}
Let $1<q<\infty$. We say that $u\in\D_0^{2,q}(\R^n)$ if there exists a sequence $(\varphi_k)_{k=1}^\infty\subset C_c^\infty(\R^n)$ such that
\begin{equation}
\Vert \D^2 u-\D^2\varphi_k\Vert_q\xrightarrow{k\rightarrow \infty} 0.
\end{equation}
Note in particular that $W^{2,q}(\R^n)\subset D_0^{2,q}(\R^n)$.

For small perturbations of constant coefficients we have the following Calderon--Zygmund estimate. See the proof of \cite[Theorem 9.11]{GT}.  

\begin{theorem}\label{th:CZperturbation}
Let $1<q<\infty$. Suppose $u\in D_0^{2,q}(\R^n)$ is a solution of (\ref{eq:nondivergenceintro}), where $\mu\in L^q(\R^n)$. There exists $\epsilon=\epsilon(n,q,\lambda)>0$ such that if $\Vert A-A_0\Vert_\infty<\epsilon$ for some constant matrix $A_0$ satisfying the same assumptions as $A$, then
\begin{equation*}
\Vert \D^2 u\Vert_{q}\leq C\Vert \mu\Vert_q,
\end{equation*}
where $C=C(n,q,\lambda)>0$.
\end{theorem}

\begin{proof}
Suppose $\Vert A-A_0\Vert_\infty<\epsilon$, where $\epsilon>0$ is chosen later. We have
\begin{equation}\label{600}
|\langle A_0,\D^2 u\rangle|=|\langle A_0-A,\D^2 u\rangle+\mu|\leq \epsilon|\D^2 u|+|\mu|
\end{equation}
Hence $\langle A_0,\D^2 u\rangle\in L^q(\R^n)$. The Calderon-Zygmund estimates for equations with uniformly elliptic constant coefficients  are well-known and follow from the Calderon--Zygmund estimates in \cite[Theorem 9.9]{GT}. Hence by (\ref{600}), we have
\begin{equation}\label{601}
\Vert \D^2 u\Vert_q\leq C\Vert\langle A_0,\D^2 u\rangle\Vert_q\leq C(\epsilon\Vert\D^2 u\Vert_q+\Vert\mu\Vert_q),
\end{equation}
where $C=C(n,q,\lambda)>0$. The claim follows by choosing $\epsilon=(2C)^{-1}$ in (\ref{601}).
\end{proof}

We consider uniformly continuous coefficients that are close to constant coefficients at infinity.  That is, assume that
\begin{equation*}
|A(x)-A(y)|\leq \omega(|x-y|), \quad\text{for every }x,y\in\R^n,
\end{equation*}
where $\omega:[0,\infty)\mapsto [0,\infty)$ is a continuous increasing function such that $\omega(0)=0$. Moreover, we assume that there exists $R>0$ and a symmetric matrix $A_0\in \R^{n\times n}$, satisfying the same bounds on the eigenvalues as $A$, such that
\begin{equation}\label{eq:infty}
\Vert A-A_0\Vert_{\infty,\R^n\setminus \overline{B(0,R)}}\leq \epsilon,
\end{equation}
where $\epsilon=\epsilon(n,q,\lambda)>0$ is given by Theorem \ref{th:CZperturbation}. 

The next estimate follows from Theorem \ref{th:CZperturbation} via a cut-off argument exactly the same way as the interior estimates of \cite[Theorem 9.11]{GT} since the coefficients can be treated locally and in $\R^n\setminus \overline{B(0,R)}$ as small perturbations of constant coefficients.
\begin{lemma}\label{th:CZinterior}
Let $1<q<\infty$. Suppose $A$ satisfies the assumptions. Let $u\in \D_0^{2,q}(\R^n)$ be a solution of  (\ref{eq:nondivergenceintro}), where $\mu\in L^q(\R^n)$. Then there exists $C=C(n,q,\lambda,\Lambda,\omega,R)>0$ such that
\begin{equation*}
\Vert\D^2 u\Vert_{q}\leq C(\Vert u\Vert_{q,B(0,2R)}+\Vert\mu\Vert_{q})
\end{equation*}
\end{lemma}

Next we assume $\mu\in L^q(\R^n)\cap L^p(\R^n)$ for some $1<q<\frac{n}{2}$ and $p>n$. We have the following existence result and Calderon--Zygmund estimates. Estimate (\ref{eq:80intro}) can be derived from the interior estimates via a compactness argument similarly as in the proof of \cite[Lemma 9.16]{GT}.  The existence of a solution can be proven by method of continuity.

\begin{theorem}\label{th:existencenondivergenceintro}
Let $1<q<\frac{n}{2}$ and $p>n$. Suppose $A$ satisfies the assumptions given above. If $\mu\in L^q(\R^n)\cap L^p(\R^n)$, there exists a unique solution $u\in D_0^{2,q}(\R^n)\cap D_0^{2,p}(\R^n)$ of (\ref{eq:nondivergenceintro}) that satisfies
\begin{equation}\label{eq:80intro}
\Vert \D^2 u\Vert_{L^q(\R^n)}+\Vert\D^2 u\Vert_{L^p(\R^n)}\leq C(\Vert \mu\Vert_{L^q(\R^n)}+\Vert\mu\Vert_{L^p(\R^n)}),
\end{equation}
where $C=C(n,q,p,\lambda,\Lambda,\omega,R)>0$.
\end{theorem}

\begin{proof}
It is enough to prove (\ref{eq:80intro}). Then the existence of a solution follows by method of continuity (see \cite[Theorem 5.2]{GT}) and the fact that the Newtonian potential of $\mu$ gives the unique solution of the corresponding Poisson equation.

Assume, for the purpose of contradiction, that (\ref{eq:80intro}) does not hold. That is, there exists a sequence of coefficient matrices $(A_k)_{k=1}^\infty$ satisfying the assumptions and a sequence functions $(u_k)_{k=1}^\infty\in D_0^{2,q}(\R^n)\cap D_0^{2,p}(\R^n)$ such that
\begin{equation}\label{620}
\Vert \D^2 u_k\Vert_{L^q(\R^n)}+\Vert\D^2 u_k\Vert_{L^p(\R^n)}=1,
\end{equation}
and
\begin{equation}\label{621}
\Vert \mu_k\Vert_{L^q(\R^n)}+\Vert\mu_k\Vert_{L^p(\R^n)}\xrightarrow{k\rightarrow\infty}0
\end{equation}
for all $k\in\N$, where $\mu_k:=\langle A_k,\D^2 u_k\rangle$. By Arzela--Ascoli theorem, we may assume that $(A_k)_{k=1}^\infty$ converges locally uniformly to a coefficient matrix function $A$ satisfying the ellipticity condition. 

Note that by interpolation and Sobolev inequality it is easy to prove that $D_0^{2,q}(\R^n)\cap D_0^{2,p}(\R^n)\subset W^{2,p}(\R^n)$ and
\begin{equation*}
\Vert v\Vert_{W^{2,p}(\R^n)}\leq C(\Vert \D^2 v\Vert_{L^q(\R^n)}+\Vert\D^2 v\Vert_{L^p(\R^n)}),
\end{equation*}
for all $v\in D_0^{2,q}(\R^n)\cap D_0^{2,p}(\R^n)$, where $C=C(n,q,p)>0$. By weak compactness of bounded sets in $W^{2,p}(\R^n)$ and Rellich--Kondrachov compactness theorem (see \cite[Theorem 7.26]{GT}), by moving to a subsequence, we may assume that $(u_k)_{k=1}^\infty$ converges to $u\in W^{2,p}(\R^n)$ weakly in $W^{2,p}(\R^n)$ and strongly in $L^{p}(B(0,2R))$. By (\ref{621}), we have
\begin{equation*}
\langle A, \D^2 u\rangle=0 \quad\text{a.e. in } \R^n.
\end{equation*}
By Morrey's theorem $u$ is continuous and $u(x)\rightarrow 0$, as $|x|\rightarrow 0$. Since $p>n$, Alexandrov maximum principle (see \cite[Theorem 9.1]{GT}) implies that $u\equiv 0$. Since $(u_k)_{k=1}^\infty$ converges strongly in $L^{p}(B(0,2R))$, we have
\begin{equation}\label{622}
\Vert u_k\Vert_{p,B(0,2R)}\xrightarrow{k\rightarrow\infty}0.
\end{equation}
By Lemma \ref{th:CZinterior}, (\ref{621}) and (\ref{622}), we have
\begin{equation*}
\Vert \D^2 u_k\Vert_{L^q(\R^n)}+\Vert\D^2 u_k\Vert_{L^p(\R^n)}\xrightarrow{k\rightarrow\infty}0.
\end{equation*}
This is a contradiction with (\ref{620}) and the claim follows.
\end{proof}

\bigskip
\noindent
Next we recall the Schauder estimates. We study non-divergence form equation (\ref{eq:nondivergenceintro}) with uniformly elliptic coefficients $A$ satisfying (\ref{eq:coef}). We have the following theorem.

\begin{theorem}\label{th:Schauderestimates}
Let $x\in\R^n$ and $R>0$. Denote $B_R:=B(x,R)$ and $B_{R/2}:=B(x,R/2)$. Let $\varphi\in C(\partial B_R)$. If $A\in C^{0,\alpha}(B_R)$ and $\mu\in C^{0,\alpha}(B_R)$ then the equation
\begin{equation*}
\begin{cases}
\langle A, \D^2 u\rangle=\mu, \quad\text{in }B_R, \\
u=\varphi \quad\text{on }\partial B_R
\end{cases}
\end{equation*}
has a unique solution $u\in C^{2,\alpha}_{loc}(B_R)\cap C(\bar{B}_R)$. Moreover, the solution $u$ satisfies the estimate
\begin{equation*}
\Vert u\Vert_{C^{2,\alpha}(B_{R/2})}\leq C\big(\Vert u\Vert_{\infty,B_R}+\Vert \mu\Vert_{C^{0,\alpha}(B_{R})}\big),
\end{equation*}
where $C=C(n,\alpha,\lambda, \Lambda, \Vert A\Vert_{C^{0,\alpha}(B_R)},R)>0$.
\end{theorem}

Theorem \ref{th:Schauderestimates}, as stated here, is  a combination of the existence theorem \cite[Lemma 6.10]{GT} for the Dirichlet problem with continuous boundary data and the interior Schauder estimates \cite[Theorem 6.2]{GT}.

\bigskip
\noindent
Finally, since we are dealing with the unbounded domain $\R^n$ we will need some control over the decay of solutions at infinity. For this purpose we have the following elementary lemma.

\begin{lemma}\label{lm:decay}
Let $u\in C^2\big(\R^n\setminus \overline{B(0,R)}\big)$ for some $R>0$ and suppose that $u$ is bounded and $u(x)\rightarrow 0$ as $|x|\rightarrow\infty$. Suppose that
\begin{equation}\label{498}
\langle A,\D^2 u\rangle=0, \quad\text{in }\R^n\setminus \overline{B(0,R)}
\end{equation}
where $A$ is a symmetric matrix-valued function such that $\lambda I\leq A\leq \Lambda I$ for some $\lambda,\Lambda>0$. If $\alpha:=\frac{\lambda}{\Lambda}(n-1)-1>0$ then there exists $C>0$ depending only on $\alpha$, $R$ and $\sup_{\partial B(0,R)} |u|$ such that
\begin{equation}\label{499}
|u(x)|\leq C|x|^{-\alpha},
\end{equation}
for all $x\in \R^n\setminus \overline{B(0,R)}$.
\end{lemma}

\begin{proof}
Define
\begin{equation*}
\varphi(x):=C|x|^{-\alpha},
\end{equation*}
where $C>0$ is chosen large enough so that
\begin{equation}\label{500}
|u(x)|\leq \varphi(x),
\end{equation}
for all $x\in\partial B(0,R)$. We note that $\D^2\varphi(x)$ has eigenvalues
\begin{equation*}
-C\alpha|x|^{-\alpha-2}\quad \text{and}\quad C\alpha(\alpha+1)|x|^{-\alpha-2}
\end{equation*}
with multiplicities $n-1$ and 1, respectively. Thus we have
\begin{equation}\label{501}
\langle A,\D^2\varphi\rangle\leq C\alpha(\Lambda(\alpha+1)-\lambda(n-1))|x|^{-\alpha-2}\leq 0.
\end{equation}

By (\ref{500}) we have
\begin{equation*}
u-\varphi\leq 0
\end{equation*}
on $\partial B(0,R)$ and by assumption we have $(u-\varphi)(x)\rightarrow 0$ as $|x|\rightarrow\infty$. By (\ref{498}), (\ref{501}) and the maximum principle $u-\varphi$ cannot have interior maximum. Hence we have
\begin{equation}\label{502}
u-\varphi\leq 0,
\end{equation}
in $\R^n\setminus \overline{B(0,R)}$. Similarly, applying maximum principle to $-u-\varphi$ gives
\begin{equation}\label{503}
-u-\varphi\leq 0,
\end{equation}
in $\R^n\setminus \overline{B(0,R)}$. Combining (\ref{502}) and (\ref{503}) gives the claimed estimate (\ref{499}).
\end{proof}

Combining Lemma \ref{lm:decay} with the following elementary interpolation lemma allows us to control the decay of the derivative of a solution as long as we can bound its second derivatives. See the proof of \cite[Lemma 6.32, p. 131]{GT} for the proof of the following elementary inequality.

\begin{lemma}\label{lm:derivativeinterpolation}
Let $\Omega\subset\R^n$ be a domain, $x\in\Omega$ and $\epsilon>0$. Suppose $u\in C_b^2(\Omega)$ and $B(x,\epsilon)\subset\Omega$. Then
\begin{equation*}
|\D u(x)|\leq \epsilon^{-1}\Vert u\Vert_{\infty,B(x,\epsilon)}+\epsilon\Vert \D^2u\Vert_{\infty,B(x,\epsilon)}.
\end{equation*}
\end{lemma}

\bigskip
\noindent

\subsection*{Schauder fixed point theorem}
One of our main tools is the classical Schauder fixed point theorem.

\begin{theorem}{(\cite[Corollary 11.2]{GT})}\label{th:fixedpoint}
Let $\mathcal{D}$ be a closed convex set in a Banach space $X$. Let $T:\mathcal{D}\rightarrow\mathcal{D}$ be a continuous mapping such that $T(\mathcal{D})$ is precompact. Then $T$ has a fixed point.
\end{theorem}


\section{\textit{A priori} derivative estimate}\label{section:Apriori}
In this section we derive  \textit{a priori} estimates for classical solutions of the equation (\ref{eq:mainequation}). Our main result is the following new derivative estimate. Although our notation is different, the proof is based on the proof  of the derivative estimate \cite[Theorem 3.5]{BS} for bounded data. We show that the method extends to data in $L^p(\R^n)$, $p>n$.

\begin{theorem}\label{th:derivativeapriori}
Let $p>n$ and let $\rho\in \mathcal{X}^*\cap L^p(\R^n)$ be continuously differentiable. Suppose $u\in D_0^{1,2}(\R^n)\cap C_b^2(\R^n)$, $|\D u|<1$ in $\R^n$, and $u$ is a classical solution of (\ref{eq:mainequation}). Then we have
\begin{equation*}
\Vert \D u\Vert_\infty\leq 1-\theta,
\end{equation*}
where $\theta=\theta(n,p,\Vert\rho\Vert_{\mathcal{X}^*},\Vert\rho\Vert_p)>0$.
\end{theorem}

Throughout this section we assume that $u\in D_0^{1,2}(\R^n)\cap C_b^2(\R^n)$ satisfies $|\D u|<1$ in $\R^n$ and $u$ satisfies equation (\ref{eq:mainequation}) pointwise in $\R^n$. We study the function $\nu\in C_b^1(\R^n)$ defined by
\begin{equation*}
\nu := \frac{1}{\sqrt{1-|\D u|^2}}.
\end{equation*}
We note that $\nu\geq 1$ by definition. By assumption, we have $\D u\in L^2(\R^n)\cap C_b^1(\R^n)$ and it follows that $\D u(x)\rightarrow 0$, as $|x|\rightarrow \infty$. Thus $\nu$ satifies the boundary condition $\nu(x)\rightarrow 1$, as $|x|\rightarrow\infty$. In particular, the continuous function $\nu$ is bounded.

The H\"older continuity of the derivative of the solution follows from classical theory once the derivative estimate of Theorem \ref{th:derivativeapriori} is established.
\begin{corollary}\label{cor:aprioriholder}
Under the assumptions of Theorem \ref{th:derivativeapriori} there exists $\alpha=\alpha(n,p,\theta)>0$ and $K=K(n,p,\theta,\Vert\rho\Vert_p)$ such that
\begin{equation*}
\sup_{x,y\in\R^n}\frac{|\D u(x)-\D u(y)|}{|x-y|^\alpha}\leq K.
\end{equation*}
\end{corollary}

\begin{proof}
As is well-known, the equation (\ref{eq:mainequation}) can be differentiated to obtain an equation for the partial derivatives of $u$. Let $\varphi\in C_c^2(\R^n)$ and let $i=1,\dots,n$. Testing (\ref{eq:mainequation}) with the function $\partial_i\varphi\in C_c^1(\R^n)$ and integrating by parts, we have
\begin{equation*}
\begin{aligned}
-\int_{\R^n}\varphi\partial_i\rho\dx &= \int_{\R^n}\rho\partial_i\varphi\dx \\
&=\int_{\R^n}\langle\D F(Du),\D\partial_i\varphi\rangle\dx \\
&=\int_{\R^n}\langle\D^2F(\D u)\D\partial_iu,\D\varphi\rangle\dx.
\end{aligned}
\end{equation*}
By approximation, the equation
\begin{equation}\label{1}
\int_{\R^n}\langle\D^2F(\D u)\D\partial_iu,\D\varphi\rangle\dx=-\int_{\R^n}\varphi\partial_i\rho\dx
\end{equation}
holds for all $\varphi\in C_c^1(\R^n)$. That is, the partial derivative $\partial_i u$ is a weak solution of the equation
\begin{equation}\label{5}
-\div\big(\D^2F(\D u)\D\partial_i u\big)=\partial_i\rho.
\end{equation}
Since $\D^2F(\D u)=\nu I+\nu^3\D u\otimes\D u$, the ellipticity bounds
\begin{equation}\label{4}
\nu|\xi|^2\leq\langle \D^2F(\D u)\xi,\xi\rangle\leq \nu^3|\xi|^2
\end{equation}
hold for all $\xi\in\R^n$. By Theorem \ref{th:derivativeapriori}, we have $1\leq\nu\leq \frac{1}{\sqrt{1-(1-\theta)^2}}$, where $\theta=\theta(n,p,\Vert\rho\Vert_{\mathcal{X}^*},\Vert\rho\Vert_p)>0$. Therefore (\ref{4}) implies that the coefficient matrix $\D^2F(\D u)$ is uniformly elliptic in $\R^n$ with an ellipticity ratio depending only on $\theta$. Thus $\partial_i u$ is a weak solution of a uniformly elliptic divergence form equation (\ref{5}). By Theorem \ref{th:C1,alpha}, we have
\begin{equation*}
[\partial_iu]_{\alpha,B_1}\leq C\big(\Vert\partial_i u\Vert_\infty+\Vert\rho\Vert_{p}\big)\leq C\big(1+\Vert\rho\Vert_{p}\big),
\end{equation*}
where $\alpha=\alpha(n,p,\theta)>0$, $C=C(n,p,\theta)>0$, and $B_1=B(x,1)$ for any $x\in\R^n$. The result follows immediately from the previous estimate since $\Vert \D u\Vert_\infty\leq 1$.
\end{proof}

Before going into the proof of Theorem \ref{th:derivativeapriori}, we point out two straightforward estimates. Testing equation (\ref{eq:mainequation}) by the solution itself, we have an $L^2$ estimate for the derivative of the solution. See \cite[Proposition 2.7]{BdAP} for a proof that applies to a minimizer $u\in\mathcal{X}$.
\begin{lemma}\label{lm:derivativeL2estimate}
Let $\rho\in \mathcal{X}^*$. Suppose $u\in D_0^{1,2}(\R^n)\cap C_b^2(\R^n)$, $|\D u|<1$ in $\R^n$, and $u$ is a classical solution of (\ref{eq:mainequation}). Then we have
\begin{equation*}
\Vert u\Vert_{D_0^{1,2}(\R^n)}\leq \Vert \rho\Vert_{\mathcal{X}^*}.
\end{equation*}
\end{lemma}

By Morrey's inequality, the family $\mathcal{X}$ is continuously embedded in $L^\infty(\R^n)$ (see \cite[Lemma 2.1]{BdAP}) and by Lemma \ref{lm:derivativeL2estimate}, we have a quantitative bound for the supremum of $|u|$ in terms of the given data $\rho$.

\begin{lemma}\label{lm:supestimate}
Let $\rho\in \mathcal{X}^*$. Suppose $u\in D_0^{1,2}(\R^n)\cap C_b^2(\R^n)$, $|\D u|<1$ in $\R^n$, and $u$ is a classical solution of (\ref{eq:mainequation}). Then we have
\begin{equation*}
\sup_{\R^n}|u|\leq M,
\end{equation*}
where $M=M(n,\Vert\rho\Vert_{\mathcal{X}^*})>0$.
\end{lemma}

Next we prove Theorem \ref{th:derivativeapriori}. The proof is divided into two main estimates given in Theorem \ref{th:nuestimate} and Theorem \ref{th:integrability}. The first estimate gives a local supremum bound for the quantity $\nu$ in terms of the $L^p$-norms of $\nu$ and $\rho$.

\begin{theorem}\label{th:nuestimate}
Let $p>n$, $x_0\in\R^n$ and $R>0$. Suppose $u\in D_0^{1,2}(\R^n)\cap C_b^2(\R^n)$, $|\D u|<1$ in $\R^n$, and $u$ is a classical solution of (\ref{eq:mainequation}). Then
\begin{equation*}
\sup_{B(x_0,R/2)}\nu\leq C\Big[\Big(\intbar_{B(x_0,R)}\nu^p\dx\Big)^\frac{n}{p(p-n)}+R^\frac{n}{p-n}\Big(\intbar_{B(x_0,R)}|\rho|^p\dx\Big)^\frac{n}{p(p-n)}\Big]\Big(\intbar_{B(x_0,R)}\nu^p\dx\Big)^\frac{1}{p},
\end{equation*}
where $C=C(n,p)>0$.
\end{theorem}
\begin{proof}
Let $\varphi\in C_c^1(\R^n)$ and suppose $\varphi\geq 0$. Using (\ref{1}), we compute that
\begin{equation}\label{2}
\begin{aligned}
\sum_{i=1}^n\int_{\R^n}\langle \D^2F(\D u)\D\partial_i u,\D\varphi\rangle\partial_i u\dx &= \sum_{i=1}^n\int_{\R^n}\varphi\partial_i u\partial_i\rho\dx \\
&\quad -\sum_{i=1}^n\int_{\R^n}\langle\D^2F(\D u)\D \partial_i u,\D \partial_i u\rangle\varphi\dx.
\end{aligned}
\end{equation}
Since $F$ is convex, we have the estimate
\begin{equation}\label{3}
\sum_{i=1}^n\int_{\R^n}\langle\D^2F(\D u)\D \partial_i u,\D \partial_i u\rangle\varphi\dx\geq 0.
\end{equation}
Denote $A:=\nu^{-3}\D^2F(\D u)=\nu^{-2}I+\D u\otimes \D u$ and note that $\D\nu=\nu^3\sum_{i=1}^n\partial_i u\D\partial_i u$. We use (\ref{3}) to estimate the right hand side of the equation (\ref{2}) and derive the estimate
\begin{equation}\label{eq:generalnupde}
\int_{\R^n} \langle A\D\nu,\D\varphi\rangle\dx\leq \int_{\R^n}\langle\D u,\D\rho\rangle\varphi\dx.
\end{equation}
It is easily verified that
\begin{equation}\label{eq:ellipticity}
\nu^{-2}|\xi|^2\leq \langle A\xi,\xi\rangle\leq |\xi|^2,
\end{equation}
for all $\xi\in\R^n$, and
\begin{equation}\label{eq:eigenvector}
A\D u = \D u.
\end{equation}

By (\ref{eq:generalnupde}) and (\ref{eq:mainequation}), we have
\begin{equation}\label{eq:5}
\begin{aligned}
\int_{\R^n} \langle A\D\nu,\D\varphi\rangle\dx &\leq \int_{\R^n} \langle\D u,\D\rho\rangle\varphi\dx \\
&=\int_{\R^n} \nu^{-1}\varphi\rho^2\dx+\int_{\R^n} \nu^{-1}\langle\D u,\D\nu\rangle\varphi\rho\dx-\int_{\R^n}\langle\D u,\D\varphi\rangle\rho\dx,
\end{aligned}
\end{equation}
for any $\varphi\in C_c^1(\R^n)$ such that $\varphi\geq 0$. Let $q>n$ and $\eta\in C_c^\infty(\R^n)$. Choosing $\varphi := \eta^2\nu^{q-1}\in C_c^1(\R^n)$ in (\ref{eq:5}), we have
\begin{equation}\label{eq:107}
\begin{aligned}
(q-1)\int_{\R^n}\eta^2\nu^{q-2}\langle A\D\nu,\D\nu\rangle\dx &\leq\int_{\R^n}\eta^2\nu^{q-2}\rho^2\dx \\
&\quad -(q-2)\int_{\R^n}\eta^2\nu^{q-2}\langle \D u,\D\nu\rangle\rho\dx \\
&\quad -2\int_{\R^n}\eta\nu^{q-1}\langle \D u,\D\eta\rangle\rho\dx \\
&\quad -2\int_{\R^n}\eta\nu^{q-1}\langle A\D\nu,\D\eta\rangle\dx.
\end{aligned}
\end{equation}
We use (\ref{eq:eigenvector}) and the Cauchy--Schwarz inequality with respect to the inner product $\langle A\cdot,\cdot\rangle$ to estimate the right hand side of (\ref{eq:107}). We have
\begin{equation}\label{eq:6}
\begin{aligned}
(q-1)\int_{\R^n} \eta^2\nu^{q-2}\langle A\D\nu,\D\nu\rangle\dx \leq & \int_{\R^n}\eta^2\nu^{q-2}\rho^2\dx \\
&+(q-2)\int_{\R^n}\eta^2\nu^{q-2}|\D u|\sqrt{\langle A\D\nu,\D\nu\rangle}|\rho|\dx \\
&+2\int_{\R^n}|\eta|\nu^{q-1}|\D u|\sqrt{\langle A\D\eta,\D\eta\rangle}|\rho|\dx \\
&+2\int_{\R^n}|\eta|\nu^{q-1}\sqrt{\langle A\D\nu,\D\nu\rangle}\sqrt{\langle A\D\eta,\D\eta\rangle}\dx.
\end{aligned}
\end{equation}
By (\ref{eq:6}) and Cauchy's inequality (the elementary inequality $ab\leq \epsilon a^2+(4\epsilon)^{-1}b^2$, for all $a,b\geq 0$ and $\epsilon>0$), we have
\begin{equation}\label{eq:6B}
\begin{aligned}
\frac{q-1}{2}\int_{\R^n} \eta^2\nu^{q-2}\langle A\D\nu,\D\nu\rangle\dx &\leq \frac{5}{q-1}\int_{\R^n}\nu^q\langle A\D\eta,\D\eta\rangle\dx \\
&\quad+\big(\frac{(q-2)^2+(q-1)^2}{q-1}+1\big)\int_{\R^n}\eta^2\nu^{q-2}\rho^2\dx.
\end{aligned}
\end{equation}
Since $q>n\geq 3$, we may divide both sides of (\ref{eq:6B}) by $(q-1)/2$ and estimate to get
\begin{equation}\label{eq:generalcaccioppoli}
\int_{\R^n} \eta^2\nu^{q-2}\langle A\D\nu,\D\nu\rangle\dx \leq \frac{10}{(q-1)^2}\int_{\R^n}\nu^q\langle A\D\eta,\D\eta\rangle\dx+5\int_{\R^n}\eta^2\nu^{q-2}\rho^2\dx.
\end{equation}

Define $\phi:=\eta\nu^\frac{q-2}{2}\in C_c^1(\R^n)$. We note that
\begin{equation}\label{eq:7}
|\D\phi|^2\leq \frac{1}{2}(q-2)^2\eta^2\nu^{q-4}|\D\nu|^2+2\nu^{q-2}|\D\eta|^2.
\end{equation}
We denote $\chi:=\frac{n}{n-2}$. By Sobolev inequality, we have
\begin{equation}\label{eq:8}
\Big(\int_{\R^n} \phi^{2\chi}\dx\Big)^\frac{1}{\chi}\leq c_0^2\int_{\R^n} |\D\phi|^2\dx,
\end{equation}
where $c_0=c_0(n)>0$. By (\ref{eq:8}), (\ref{eq:7}), (\ref{eq:generalcaccioppoli}) and (\ref{eq:ellipticity}) it follows that
\begin{equation}\label{eq:9}
\begin{aligned}
\Big(\int_{\R^n} \eta^{2\chi}\nu^{(q-2)\chi}\dx\Big)^\frac{1}{\chi} &\leq \frac{c_0^2(q-2)^2}{2}\int_{\R^n} \eta^2\nu^{q-4}|\D\nu|^2\dx+2c_0^2\int_{\R^n} \nu^{q-2} |\D\eta|^2\dx \\
&\leq 7c_0^2\int_{\R^n} \nu^q |\D\eta|^2\dx+3c_0^2(q-2)^2\int_{\R^n} \eta^2\nu^{q-2}\rho^2\dx.
\end{aligned}
\end{equation}

Let $m=0,1,\dots$. We define
\begin{equation*}
R_m:=\big(2^{-1}+2^{-(m+1)}\big)R \quad\text{and}\quad  B_m:=B(x_0,R_m).
\end{equation*}
Let $\eta_m\in C_c^\infty(B_m)$ be such that $0\leq \eta_m\leq 1$, $\eta_m\equiv 1$ in $B_{m+1}$, and $|\D\eta_m|\leq \frac{2}{R_m-R_{m+1}}\leq 2^{m+3}R^{-1}$. Choosing the cut-off function $\eta_m$ in (\ref{eq:9}), we have
\begin{equation}\label{eq:16}
\begin{aligned}
\Big(\intbar_{B_{m+1}}\nu^{(q-2)\chi}\dx\Big)^\frac{1}{\chi} \leq& \ 4^{m+n+3}c_0^2\intbar_{B_m} \nu^q\dx \\
&+4^{n}c_0^2(q-2)^2R^2\intbar_{B_m}\nu^{q-2}\rho^2\dx.
\end{aligned}
\end{equation}
We estimate the terms on the right hand side of (\ref{eq:16}) separately. Define $\alpha:=\frac{p}{p-2}$. Since $p>n$, we have $\alpha<\chi$. In what follows we assume $q\geq p$. Hence we may apply H\"older's inequality to conclude
\begin{equation}\label{eq:cutoffterm}
\intbar_{B_m} \nu^q\dx\leq \Big(\intbar_{B_m}\nu^p\dx\Big)^\frac{2}{p}\Big(\intbar_{B_m} \nu^{\alpha(q-2)}\dx\Big)^\frac{1}{\alpha},
\end{equation}
and similarly we have
\begin{equation}\label{eq:curvatureterm}
\intbar_{B_m}\nu^{q-2}\rho^2\dx\leq \Big(\intbar_{B_m}|\rho|^p\dx\Big)^\frac{2}{p}\Big(\intbar_{B_m}\nu^{\alpha(q-2)}\dx\Big)^\frac{1}{\alpha}
\end{equation}

By (\ref{eq:16}), (\ref{eq:cutoffterm}) and (\ref{eq:curvatureterm}), we have
\begin{equation}\label{eq:generalhigherintegrability}
\Big(\intbar_{B_{m+1}}\nu^{(q-2)\chi}\dx\Big)^\frac{1}{(q-2)\chi} \leq 4^\frac{m}{q-2}\tau^\frac{2}{q-2}(q-2)^\frac{2}{q-2}\Big(\intbar_{B_m} \nu^{\alpha(q-2)}\dx\Big)^\frac{1}{\alpha(q-2)},
\end{equation}
where
\begin{equation}\label{eq:deftau}
\tau^2:=4^{2n+3}c_0^2\Big(\intbar_{B(x_0,R)}\nu^p\dx\Big)^\frac{2}{p}+4^{2n}c_0^2R^2\Big(\intbar_{B(x_0,R)}|\rho|^p\dx\Big)^\frac{2}{p}.
\end{equation}
We define $\beta:=\frac{\chi}{\alpha}>1$ and
\begin{equation*}
q_0:=p,\quad q_m:=\beta^mp \quad\text{and}\quad M_m:=\Big(\intbar_{B_m}\nu^{q_m}\dx\Big)^\frac{1}{q_m}.
\end{equation*}
By (\ref{eq:generalhigherintegrability}), we have
\begin{equation}\label{eq:generaliteration}
M_{m+1}\leq (2\beta)^{\frac{2}{p-2}m\beta^{-m}}\big((p-2)\tau\big)^{\frac{2}{p-2}\beta^{-m}}M_m
\end{equation}
Iterating (\ref{eq:generaliteration}) and taking the limit as $m\rightarrow\infty$, we have
\begin{equation}\label{eq:iterationlimit}
\lim_{m\rightarrow\infty}M_m\leq (2\beta)^{\frac{2}{p-2}\sum_{j=1}^\infty j\beta^{-j}}\big((p-2)\tau\big)^{\frac{2}{p}\sum_{j=0}^\infty\beta^{-j}}M_0
\end{equation}
As $\beta>1$ we have
\begin{equation}\label{eq:17}
\frac{2}{p-2}\sum_{j=0}^\infty\beta^{-j}=\frac{n}{p-n} \quad\text{and}\quad \frac{2}{p-2}\sum_{j=1}^\infty j\beta^{-j}=\frac{pn(n-2)}{2(p-n)^2}.
\end{equation}
By (\ref{eq:iterationlimit}), (\ref{eq:17}) and (\ref{eq:deftau}), we have
\begin{equation*}
\sup_{B(x,R/2)}\nu\leq C\Big[\Big(\intbar_{B(x,R)}\nu^p\dx\Big)^\frac{n}{p(p-n)}+R^\frac{n}{p-n}\Big(\intbar_{B(x,R)}|\rho|^p\dx\Big)^\frac{n}{p(p-n)}\Big]\Big(\intbar_{B(x,R)}\nu^p\dx\Big)^\frac{1}{p},
\end{equation*}
where $C=C(n,p)>0$. This completes the proof.
\end{proof}

Since $\nu(x)\rightarrow 1$, as $|x|\rightarrow\infty$, the function $(\nu-k)_+:=\max(\nu-k,0)$ is compactly supported for any $k>1$. The $L^p$-norm of $(\nu-k)_+$ can be controlled globally in terms of the $L^p$-norm of $\rho$ and the $L^\infty$-norm of $u$.

\begin{theorem}\label{th:integrability}
Let $p>n$ and $k>1$. Suppose $u\in D_0^{1,2}(\R^n)\cap C_b^2(\R^n)$, $|\D u|<1$ in $\R^n$, and $u$ is a classical solution of (\ref{eq:mainequation}). Then
\begin{equation}\label{eq:pintegrability}
\Big(\int_{\R^n} (\nu-k)_+^p\dx\Big)^\frac{1}{p}\leq CM\Big(\int_{\R^n} |\rho|^p\dx\Big)^\frac{1}{p},
\end{equation}
where $C=C(p,k)>0$ and $M=M(n,\Vert\rho\Vert_{\mathcal{X}^*})>0$ is the supremum bound for $u$ given in Lemma \ref{lm:supestimate}.
\end{theorem}

\begin{proof}
Using $\varphi=(\nu-k)_+^{p-1}\in C_c^1(\R^n)$ as a testing function in (\ref{eq:generalnupde}) and arguing similarly as in the proof of (\ref{eq:generalcaccioppoli}), it is possible to derive the estimate
\begin{equation}\label{eq:pcaccioppoli}
\int_{\R^n}(\nu-k)_+^{p-2}\langle A\D\nu,\D\nu\rangle\dx\leq 2\int_{\R^n}(\nu-k)_+^{p-2}|\rho|^2\dx.
\end{equation} 
By (\ref{eq:mainequation}), we have
\begin{equation}\label{eq:weakcurvature}
\int_{\R^n}\nu\langle\D u,\D\varphi\rangle\dx=\int_{\R^n}\varphi\rho\dx,
\end{equation}
for any $\varphi\in C_c^1(\R^n)$. Choosing $\varphi=u\nu^{-1}(\nu-k)_+^p\in C_c^1(\R^n)$ in (\ref{eq:weakcurvature}), we have
\begin{equation}\label{eq:108}
\begin{aligned}
\int_{\R^n}(\nu-k)_+^p|\D u|^2\dx &= \int_{\R^n}u\nu^{-1}(\nu-k)_+^p\langle\D u,\D\nu\rangle\dx \\
&\quad -p\int_{\R^n}u(\nu-k)_+^{p-1}\langle\D u,\D\nu\rangle\dx \\
&\quad +\int_{\R^n}u\nu^{-1}(\nu-k)_+^p\rho\dx
\end{aligned}
\end{equation}
By (\ref{eq:108}), (\ref{eq:eigenvector}) and the Cauchy--Schwarz inequality with respect to the inner product $\langle A\cdot,\cdot\rangle$, we have
\begin{equation}\label{eq:10}
\begin{aligned}
\int_{\R^n}(\nu-k)_+^p|\D u|^2\dx\leq& p\int_{\R^n}|u|(\nu-k)_+^{p-1}|\D u|\sqrt{\langle A\D\nu,\D\nu\rangle}\dx \\
&+\int_{\R^n}|u|(\nu-k)_+^{p-1}|\rho|\dx
\end{aligned}
\end{equation}
By Lemma \ref{lm:supestimate}, we have $|u|\leq M$. We note that
\begin{equation}\label{eq:epsilonk}
\nu\geq k \quad \implies \quad |Du|^2\geq 1-\frac{1}{k^2}=:\epsilon_k.
\end{equation}
Using (\ref{eq:epsilonk}) and Cauchy's inequality in (\ref{eq:10}), we have
\begin{equation}\label{eq:11}
\begin{aligned}
\int_{\R^n}(\nu-k)_+^p\dx\leq& \frac{2p^2M^2}{\epsilon_k^2}\int_{\R^n}(\nu-k)_+^{p-2}\langle A\D\nu,\D\nu\rangle\dx \\
&+\frac{2M^2}{\epsilon_k^2}\int_{\R^n}(\nu-k)_+^{p-2}\rho^2\dx.
\end{aligned}
\end{equation}
By (\ref{eq:11}), (\ref{eq:pcaccioppoli}) and (\ref{eq:ellipticity}), we have
\begin{equation}\label{eq:13}
\int_{\R^n}(\nu-k)_+^p\dx\leq\frac{6p^2M^2}{\epsilon_k^2}\int_{\R^n}(\nu-k)_+^{p-2}\rho^2\dx.
\end{equation}
By H\"older's inequality, we have
\begin{equation}\label{eq:14}
\int_{\R^n}(\nu-k)_+^{p-2}\rho^2\dx\leq \Big(\int_{\R^n}|\rho|^p\dx\Big)^\frac{2}{p}\Big(\int_{\R^n}(\nu-k)_+^p\dx\Big)^\frac{p-2}{p}.
\end{equation}
By (\ref{eq:13}) and (\ref{eq:14}), we have
\begin{equation}
\int_{\R^n}(\nu-k)_+^p\dx\leq \frac{6p^2M^2}{\epsilon_k^2}\Big(\int_{\R^n}|\rho|^p\dx\Big)^\frac{2}{p}\Big(\int_{\R^n}(\nu-k)_+^p\dx\Big)^\frac{p-2}{p}.
\end{equation}
The claimed inequality (\ref{eq:pintegrability}) follows immediately.
\end{proof}

We combine the estimates of Theorem \ref{th:nuestimate} and Theorem \ref{th:integrability} to prove Theorem \ref{th:derivativeapriori}.

\begin{proof}[Proof of Theorem \ref{th:derivativeapriori}]
By Theorem \ref{th:integrability}, we have
\begin{equation*}
\Big(\int_{\R^n}(\nu-2)_+^p\dx\Big)^\frac{1}{p}\leq CM\Big(\int_{\R^n}|\rho|^p\dx\Big)^\frac{1}{p},
\end{equation*}
where $C=C(n,p)>0$ and $M=M(n,\Vert\rho\Vert_\frac{2n}{n+2})>0$. Let $x_0\in \R^n$. By triangle inequality, it follows that
\begin{equation}\label{eq:202}
\Big(\intbar_{B(x_0,1)}\nu^p\dx\Big)^\frac{1}{p}\leq \Big(\intbar_{B(x_0,1)}\big((\nu-2)_++2\big)^p\dx\Big)^\frac{1}{p}\leq CM\Big(\int_{\R^n}|\rho|^p\dx\Big)^\frac{1}{p}+2
\end{equation}
By Theorem \ref{th:nuestimate} and (\ref{eq:202}) we have
\begin{equation*}
\sup_{\R^n}\nu\leq C,
\end{equation*}
where $C=C(n,p,\Vert\rho\Vert_{\mathcal{X}^*},\Vert\rho\Vert_p)>0$.
\end{proof}


\section{Proof of Theorem \ref{th:main}}\label{section:proof}
In this section we prove Theorem \ref{th:main} using Leray--Schauder fixed point theorem and the \textit{a priori} derivative estimate of Theorem \ref{th:derivativeapriori}. The method is standard in the theory of quasilinear equations on bounded domains, see for example \cite[Chapter 11]{GT}. We follow the main ideas of \cite[Theorem 3.6]{BS} where the method was applied to Lorentz mean curvature equation on bounded domains. Our treatment of the problem in the unbounded domain $\R^n$ is an adaptation of the standard method. We replace boundary conditions with suitable decay estimates which allows us to recover compactness of the fixed point mapping despite the unboundedness of the domain. In addition, we use Calderon-Zygmund estimates and Sobolev embedding theorem to obtain global control over the oscillations of the solutions. Schauder estimates can be used locally as in the case of bounded domains once the oscillations of solutions are bounded globally.

The following lemma is a restatement of basic results for non-divergence form equations in our setting.

\begin{lemma}\label{lm:existence}
Let $\alpha\in(0,1]$, $p>n$, $R>0$, and $\theta>0$. Let $v\in C^{1,\alpha}(\R^n)$ satisfy $\Vert\D v\Vert_\infty\leq 1-\theta$. Furthermore, suppose $\Vert \D v\Vert_{\infty,\R^n\setminus \overline{B(0,R)}}\leq \epsilon$, where $\epsilon=\epsilon(n,p)>0$ is such that
\begin{equation}\label{eq:103}
\sup_{\xi\in B(0,\epsilon)}| \D^2F(\xi)-I|\leq \epsilon_0,
\end{equation}
where $\epsilon_0=\epsilon_0(n,p)>0$ is given by Theorem \ref{th:CZperturbation}. Suppose $\mu\in L^\frac{2n}{n+2}(\R^n)\cap L^p(\R^n)\cap C^{0,1}(\R^n)$. Then there exists a unique solution  $u\in C^{2,\alpha}(\R^n)$ of the linear equation
\begin{equation}\label{eq:auxpde}
Q_vu:=-\langle\D^2F(\D v),\D^2 u\rangle = \mu,
\end{equation}
satisfying $u(x)\rightarrow 0$, as $|x|\rightarrow 0$. Furthermore, we have
\begin{equation}
\Vert u\Vert_{C^{2,\alpha}(\R^n)}\leq C,
\end{equation}
where $C=C(n,p,\theta,\alpha,R,\Vert v\Vert_{C^{1,\alpha}(\R^n)},\Vert \mu\Vert_\frac{2n}{n+2},\Vert\mu\Vert_p,\Vert \mu\Vert_{C^{0,1}(\R^n)})>0$.

\end{lemma}
\begin{proof}
Since $|\D v|\leq 1-\theta$ in $\R^n$, we have
\begin{equation}\label{eq:101}
\big(1-(1-\theta)^2\big)^{-\frac{1}{2}} I\leq \D^2F(\D v)\leq \big(1-(1-\theta)^2\big)^{-\frac{3}{2}} I \quad\text{in }\R^n,
\end{equation}
The mapping $\D^2F$ is Lipschitz continuous and bounded on $B(0,1-\theta)$. Since $\D v\in C^{0,\alpha}(\R^n)$ and $\D v$ takes values in $B(0,1-\theta)$, it follows that $\D^2F(\D v)\in C^{0,\alpha}(\R^n)$ and
\begin{equation}\label{eq:102}
\Vert\D^2F(\D v)\Vert_{C^{0,\alpha}(\R^n)}\leq C,
\end{equation}
where $C=C(\theta, \Vert v\Vert_{C^{1,\alpha}(\R^n)})>0$. 

By (\ref{eq:101}), (\ref{eq:102}) and (\ref{eq:103}) the assumptions of Theorem \ref{th:existencenondivergenceintro} and Theorem \ref{th:Schauderestimates} are satisfied. Hence there exists a unique solution to (\ref{eq:auxpde}) that satisfies 
\begin{equation}\label{eq:104}
\Vert \D^2 u\Vert_{\frac{2n}{n+2}}+\Vert\D^2 u\Vert_{p}\leq C(\Vert \mu\Vert_{\frac{2n}{n+2}}+\Vert\mu\Vert_{p}),
\end{equation}
where $C=C(n,p,\theta,\Vert v\Vert_{C^{1,\alpha}(\R^n)},R)>0$. Note that Morrey's inequality and (\ref{eq:104}) give an $L^\infty$ estimate for $u$. Then, by Theorem \ref{th:Schauderestimates},we have $u\in C^{2,\alpha}(\R^n)$ and $u$ satisfies
\begin{equation}\label{eq:105}
\Vert u\Vert_{C^{2,\alpha}(\R^n)}\leq C,
\end{equation}
where $C=C(n,\alpha,\theta,R,\Vert v\Vert_{C^{1,\alpha}(\R^n)},\Vert \mu\Vert_\frac{2n}{n+2},\Vert\mu\Vert_p,\Vert \mu\Vert_{C^{0,1}(\R^n)})>0$.

\end{proof}

Now we are ready to apply Schauder fixed point theorem to find solutions for smooth and compactly supported data. The proof essentially follows \cite[Section 11.2]{GT} with some notable differences. The proof is divided into to three steps, the first of which is unique to the problem in unbounded domain. In the first step we prove a uniform decay of all "nice" solutions corresponding to a certain one-parameter family of data. In the second step decay estimates have to be used to to prove the continuity and the compactness of the fixed point mapping. Furthermore, the first constraint in (\ref{870}) is specific to the Lorentz mean curvature operator and the second one for the unbounded domain. Otherwise the second and the third step follw the ideas presented in \cite[Section 11.2]{GT}.

\begin{theorem}\label{th:smoothexistence}
Let $p>n$ and suppose $\rho\in C_c^\infty(\R^n)$. Then there exist $\alpha=\alpha(n,p,\Vert\rho\Vert_{\mathcal{X}^*},\Vert\rho\Vert_p)\in (0,1]$ and $u\in C^{1,\alpha}(\R^n)$ that solves (\ref{eq:mainequation}) and satisfies
\begin{equation*}
\Vert u\Vert_{C^{1,\alpha}(\R^n)}\leq K,
\end{equation*}
where $K=K(n,p,\Vert\rho\Vert_{\mathcal{X}^*},\Vert\rho\Vert_p)>0$. Moreover, we have
\begin{equation*}
\Vert\D u\Vert_\infty\leq 1-\theta,
\end{equation*}
where $\theta=\theta(n,p,\Vert\rho\Vert_{\mathcal{X}^*},\Vert\rho\Vert_p)>0$.
\end{theorem}

\begin{proof}
Let $S$ be the set of parameters $\tau\in[0,1]$ such that there exists $u_\tau\in D_0^{1,2}(\R^n)\cap C_b^2(\R^n)$ satisfying $|\D u|< 1$ in $\R^n$, and
\begin{equation}\label{eq:81}
-\div\Big({\frac{\D u_\tau}{\sqrt{1-|\D u_\tau|^2}}}\Big)=\tau\rho \quad\text{in }\R^n.
\end{equation}
Note that $0\in S$, since obviously $u_0\equiv 0$ is a solution. Hence we have $S\not=\emptyset$. By Theorem \ref{th:derivativeapriori} there exists $\theta=\theta(n,p,\Vert \rho\Vert_p,\Vert \rho\Vert_{\mathcal{X}^*})>0$ such that
\begin{equation}\label{eq:88}
\Vert\D u_\tau\Vert_\infty\leq 1-\theta,
\end{equation}
for each $\tau\in S$. By Corollary \ref{cor:aprioriholder}, Lemma \ref{lm:derivativeL2estimate}, and Lemma \ref{lm:supestimate}, there exists $\alpha=\alpha(n,p,\theta)$ and $K=K(n,p,\theta,\Vert\rho\Vert_{\mathcal{X}^*},\Vert\rho\Vert_p)>0$, such that
\begin{equation}\label{eq:106}
\Vert u_\tau\Vert_{C^{1,\alpha}(\R^n)}+\Vert u_\tau\Vert_{D_0^{1,2}(\R^n)}\leq K,
\end{equation}
for each $\tau\in S$.

\textit{Step 1. } We show that the derivatives $\D u_\tau$ have a uniform decay rate at infinity for all $\tau\in S$. That is, we prove that
\begin{equation}\label{23}
\omega_\infty(R):=\sup\{|\D u_\tau(x)|:\tau\in S,x\in\R^n\setminus \overline{B(0,R)}\}\rightarrow 0,
\end{equation}
as $R\rightarrow \infty$. We first prove two simple claims.

\textit{Claim.} Let $\tau_0\in S$ and $\epsilon>0$. There exists $R_{\epsilon,\tau_0}>0$ such that 
\begin{equation}\label{21}
\Vert\D u_{\tau_0}\Vert_{\infty,\R^n\setminus{\overline {B(0,R_{\epsilon,\tau_0})}}}\leq \epsilon.
\end{equation}

Since $u_{\tau_0}\in D_0^{1,2}(\R^n)\cap C^{1,\alpha}(\R^n)$, we have $\lim_{|x|\rightarrow\infty}\D u_{\tau_0}(x)=0$. This proves the claim.

\textit{Claim.} Let $\epsilon>0$. There exists $\delta_\epsilon>0$ such that
\begin{equation}\label{20}
\Vert \D u_{\tau_2}-\D u_{\tau_1}\Vert_\infty\leq \epsilon,
\end{equation}
for all $\tau_1,\tau_2\in S$ such that $|\tau_2-\tau_1|\leq\delta_\epsilon$. 

Suppose $\tau_1,\tau_2\in S$. Testing the equation (\ref{eq:81}) with $u_{\tau_2}-u_{\tau_1}$, we have
\begin{equation}\label{eq:122}
\int_{\R^n}\langle \D F(\D u_{\tau_2})-\D F(\D u_{\tau_1}),\D u_{\tau_2}-\D u_{\tau_1}\rangle\dx=(\tau_2-\tau_1)\int_{\R^n}\rho(u_{\tau_2}-u_{\tau_1})\dx.
\end{equation}
The vector field $\D F$ satisfies the monotonicity estimate
\begin{equation}\label{eq:123}
\int_{\R^n}\langle \D F(\D u_{\tau_2})-\D F(\D u_{\tau_1}),\D u_{\tau_2}-\D u_{\tau_1}\rangle\dx\geq \int_{\R^n}|\D u_{\tau_2}-\D u_{\tau_1}|^2\dx.
\end{equation}
We have
\begin{equation}\label{eq:124}
\big|\int_{\R^n}\rho(u_{\tau_2}-u_{\tau_1})\dx\big| \leq \Vert \rho\Vert_{\mathcal{X}^*}\Vert \D u_{\tau_1}-\D u_{\tau_2}\Vert_2
\end{equation}
By (\ref{eq:122}), (\ref{eq:123}) and (\ref{eq:124}) we have
\begin{equation}\label{eq:87}
\Vert \D u_{\tau_2}-\D u_{\tau_1}\Vert_2\leq |\tau_1-\tau_2|\Vert \rho\Vert_{\mathcal{X}^*}.
\end{equation}
By (\ref{eq:106}) the functions $\D u_\tau$, $\tau\in S$ are uniformly bounded in $C^{0,\alpha}(\R^n)$ and hence (\ref{eq:87}) implies (\ref{20}) for $\delta_\epsilon>0$ small enough. This finishes the proof of the claim.

Now we prove (\ref{23}) by a simple covering argument. Let $\epsilon>0$. Let $\delta_\epsilon>0$ be as in (\ref{20}). Since the closure of $S$ is compact and $\bar{S}\subset\bigcup_{\tau\in S}(\tau-\delta_\epsilon,\tau+\delta_\epsilon)$, there exists $N_\epsilon\in\N$ and a finite sequence $(\tau_i)_{i=1}^{N_\epsilon}\subset S$ such that the open intervals $(\tau_i-\delta_{\epsilon},\tau_i+\delta_{\epsilon})$, $i=1,\dots,N_\epsilon$, cover $S$. Define 
\begin{equation*}R_\epsilon=\max_{1\leq i\leq N_\epsilon}R_{\epsilon,\tau_i},
\end{equation*}
where $R_{\epsilon,\tau_i}>0$ is as in (\ref{21}), for each $i=1,\dots,N_\epsilon$.

Let $\tau\in S$. There exists $1\leq i\leq N_\epsilon$ such that $|\tau-\tau_i|<\delta_\epsilon$. By (\ref{20}) and (\ref{21}), we have
\begin{equation*}
\Vert\D u_\tau\Vert_{\infty,\R^n\setminus \overline{B(0,R_\epsilon)}}\leq \Vert\D u_\tau-\D u_{\tau_i}\Vert_{\infty,\R^n\setminus \overline{B(0,R_\epsilon)}}+\Vert\D u_{\tau_i}\Vert_{\infty,\R^n\setminus \overline{B(0,R_\epsilon)}}\leq 2\epsilon.
\end{equation*}
Thus we have $\omega_\infty(R_\epsilon)\leq 2\epsilon$. We have proven that $\omega_\infty(R)\rightarrow 0$ as $R\rightarrow\infty$. This finishes step 1.

\textit{Step 2.} In this step we define a mapping $\tilde{T}$ and show that it satisfies the assumptions of Schauder fixed point theorem (Theorem \ref{th:fixedpoint}).

Let $R>0$ be chosen such that 
\begin{equation}\label{872}
\omega_\infty(R)\leq \epsilon/2,
\end{equation}
where $\omega_\infty$ is defined in (\ref{23}) and $\epsilon=\epsilon(n,p)>0$ is given in Lemma \ref{lm:existence}. Define
\begin{equation}\label{870}
\mathcal{D}=\big\{ v\in C^{1,\alpha}(\R^n):\Vert\D v\Vert_\infty\leq 1-\frac{1}{2}\theta,\Vert \D v\Vert_{\infty, \R^n\setminus \overline{B(0,R)}}\leq \epsilon\big\},
\end{equation}
where $\theta>0$ is given in (\ref{eq:88}). The set $\mathcal{D}$ is always considered as a subset of the normed space $C^{1,\alpha}(\R^n)$ with the induced topology. For each $v\in \mathcal{D}$ we define $T(v)$ to be the unique solution $u\in C^{2,\alpha\beta}(\R^n)$ to the equation
\begin{equation}
Q_vu=\rho,
\end{equation}
given by Lemma \ref{lm:existence}. 

\textit{Claim.} The mapping $T:\mathcal{D}\rightarrow C^{1,\alpha}(\R^n)$ is compact, that is, it maps bounded sets to precompact sets. 

Let $(v_k)_{k=1}^\infty\subset \mathcal{D}$ be a bounded sequence. By Lemma \ref{lm:existence}, the mapping $T:\mathcal{D}\rightarrow C^{2,\alpha}(\R^n)$ maps bounded sets to bounded sets. Hence, the sequence $(T(v_k))_{k=1}^\infty$ is bounded in $C^{2,\alpha}(\R^n)$ and consequently also in $C_b^2(\R^n)$. Moreover, by the decay estimate of Lemma \ref{lm:decay} and the interpolation inequality of Lemma \ref{lm:derivativeinterpolation} there exists a decreasing function $\gamma:[0,\infty)\rightarrow [0,\infty)$, independent of $k$, such that
\begin{equation}\label{75}
\Vert T(v_k)\Vert_{\infty,\R^n\setminus \overline{B(0,R)}}+\Vert \D T(v_k)\Vert_{\infty,\R^n\setminus \overline{B(0,R)}}\leq \gamma(R)\xrightarrow{R\rightarrow\infty}0,
\end{equation}
for all $R>0$ and all $k\in\N$. By Arzela-Ascoli theorem there exists $u\in C^{1,\alpha}(\R^n)$ such that, up to a subsequence, $T(v_k)\rightarrow u$ and $\D T(v_k)\rightarrow \D u$ uniformly on compact subsets of $\R^n$ as $k\rightarrow\infty$. By (\ref{75}), it holds that $T(v_k)\rightarrow u$ in $C_b^1(\R^n)$ as $k\rightarrow\infty$. Since the sequence $(T(v_k))_{k=1}^\infty$ is bounded in $C_b^2(\R^n)$ and converges in $C_b^1(\R^n)$, we conclude, by interpolation, that $T(v_k)\rightarrow u$ in $C^{1,\alpha}(\R^n)$ as $k\rightarrow\infty$. We have proven that the mapping $T:\mathcal{D}\rightarrow C^{1,\alpha}(\R^n)$ is compact. 

\textit{Claim.} The mapping $T:\mathcal{D}\rightarrow C^{1,\alpha}(\R^n)$ is continuous. 

Let $(v_k)_{k=1}^\infty\subset \mathcal{D}$ be a sequence such that $v_k\rightarrow v\in \mathcal{D}$ in norm as $k\rightarrow\infty$. We have
\begin{equation}\label{50}
\begin{aligned}
Q_{v}(T(v)-T(v_k)) &=Q_{v}T(v)-Q_{v_k}T(v_k)+(Q_{v_k}-Q_{v})T(v_k) \\
&=-\langle \D^2 F(\D v_k)-\D^2 F(\D v),\D^2 T(v_k)\rangle=:f_k.
\end{aligned}
\end{equation}
By Lemma \ref{lm:existence}, the sequence $(\D^2 T(v_k))_{k=1}^\infty$ is bounded in $L^\frac{2n}{n+2}(\R^n)\cap L^p(\R^n)$. Since $\D^2 F$ is Lipschitz continuous in $B(0,1-\theta)$ and $\D v_k\rightarrow \D v$ uniformly in $\R^n$ as $k\rightarrow\infty$, we have $\D^2F(\D v_k)\rightarrow \D^2F(\D v)$ uniformly in $\R^n$ as $k\rightarrow\infty$. It follows that
\begin{equation}\label{51}
f_k:= -\langle \D^2 F(\D v_k)-\D^2 F(\D v),\D^2 T(v_k)\rangle\xrightarrow{k\rightarrow\infty} 0, \quad\text{in }L^\frac{2n}{n+2}(\R^n)\cap L^p(\R^n).
\end{equation}
By (\ref{50}), (\ref{51}), and the Calderon-Zygmund estimate (\ref{eq:104}), we have
\begin{equation}\label{eq:125}
\D^2T(v_k)\xrightarrow{k\rightarrow\infty} \D^2T(v) \quad\text{in } L^\frac{2n}{n+2}(\R^n)\cap L^p(\R^n).
\end{equation}
By (\ref{eq:125}) and the fact that $(T(v_k))_{k=1}^\infty$ is bounded in $C^{2,\alpha}(\R^n)$, we have
\begin{equation}\label{eq:127}
T(v_k)\xrightarrow{k\rightarrow\infty} T(v) \quad\text{in }C_b^2(\R^n).
\end{equation}
By (\ref{eq:127}), we have
\begin{equation*}
T(v_k)\xrightarrow{k\rightarrow\infty} T(v) \quad\text{in }C^{1,\alpha}(\R^n).
\end{equation*}
Hence, the mapping $T:\mathcal{D}\rightarrow C^{1,\alpha}(\R^n)$ is continuous. This finishes the proof of the claim.

We define
\begin{equation}
\tilde{\mathcal{D}}=\mathcal{D}\cap\big\{v\in C^{1,\alpha}(\R^n):\Vert v\Vert_{C^{1,\alpha}(\R^n)}\leq 2K\big\},
\end{equation}
where $K>0$ is given in (\ref{eq:106}). Again, the set $\tilde{\mathcal{D}}$ is always considered as a subset of the normed space $C^{1,\alpha}(\R^n)$ with the induced topology. We define the mapping $\tilde{T}:\tilde{\mathcal{D}}\mapsto C^{1,\alpha}(\R^n)$ by
\begin{equation}
\tilde{T}(v)=\tau(v)T(v),
\end{equation}
where
\begin{equation}
\tau(v):=\min\Big\{1,\frac{1-\frac{1}{2}\theta}{\Vert \D T(v)\Vert_\infty},\frac{2K}{\Vert T(v)\Vert_{C^{1,\alpha}(\R^n)}},\frac{\epsilon}{\Vert DT(v)\Vert_{\infty,\R^n\setminus \overline{B(0,R)}}}\Big\}.
\end{equation}
By the definition of $\tilde{T}$ it is obvious that $\tilde{T}(\tilde{\mathcal{D}})\subset \tilde{\mathcal{D}}$. We note that $\tau:\mathcal{D}\rightarrow [0,1]$ is clearly continuous. Since $T$ is compact and continuous and $\tau$ is continuous and bounded, the mapping $\tilde{T}$ is also compact and continuous. We also note that $\tilde{\mathcal{D}}$ is a bounded, closed and convex subset of the Banach space $C^{1,\alpha}(\R^n)$. Hence $\tilde{T}$ satisfies the assumptions of the Schauder fixed point teorem (Theorem \ref{th:fixedpoint}).

\textit{Step 3.} By Schauder fixed point theorem (Theorem \ref{th:fixedpoint}) the mapping $\tilde{T}$ has a fixed point $u\in \tilde{\mathcal{D}}$. In this step we prove that $u$ is actually a fixed point of $T$ and we conclude the proof of Theorem \ref{th:smoothexistence}.

We have $u=\tilde{T}(u)=\tau(u)T(u)\in C^{2,\alpha}(\R^n)$. Hence $u$ solves the equation (\ref{eq:81}) with $\tau=\tau(u)$. Since $u\in\tilde{\mathcal{D}}$, we have $\Vert\D u\Vert_\infty\leq 1-\theta/2<1$. Moreover, by Calderon--Zygmund estimates (\ref{eq:104}) and Sobolev inequality we have $u\in D_0^{1,2}(\R^n)$. Hence $u$ satisfies the a priori estimates (\ref{eq:88}) and  (\ref{eq:106}). Furthermore, recalling the decay estimate (\ref{23}) and the choice of $R>0$ (\ref{872}), we have
\begin{equation}\label{873}
\Vert\D u\Vert_{\infty,\R^n\setminus\overline{B(0,R)}}\leq\epsilon/2.
\end{equation}

Next we show that $\tau(u)=1$. Assume for the purpose of contradiction that $\tau(u)<1$. Then by the definition of $\tau(u)$, we have
\begin{equation}\label{874}
\begin{aligned}
&\tau(u)=\frac{1-\frac{1}{2}\theta}{\Vert \D T(u)\Vert_\infty}<1, \\ 
\text{or}\quad &\tau(u)=\frac{2K}{\Vert T(u)\Vert_{C^{1,\alpha}(\R^n)}}<1, \\
\text{or}\quad &\tau(u)=\frac{\epsilon}{\Vert DT(u)\Vert_{\infty,\R^n\setminus \overline{B(0,R)}}}<1.
\end{aligned}
\end{equation}
Since $u$ is a fixed point of $\tilde{T}$, the equalities in (\ref{874}) imply, correspondingly, that
\begin{equation*}
\begin{aligned}
&\Vert\D u\Vert_\infty=\Vert\D \tilde{T}(u)\Vert_\infty=\tau(u)\Vert \D T(u)\Vert_\infty=1-\frac{1}{2}\theta>1-\theta \\
\text{or}\quad &\Vert u\Vert_{C^{1,\alpha}(\R^n)}=\Vert \tilde{T}(u)\Vert_{C^{1,\alpha}(\R^n)}=\tau(u)\Vert T(u)\Vert_{C^{1,\alpha}(\R^n)}=2K>K \\
\text{or}\quad &\Vert\D u\Vert_{\infty,\R^n\setminus \overline{B(0,R)}}=\Vert\D \tilde{T}(u)\Vert_{\infty,\R^n\setminus \overline{B(0,R)}}=\epsilon>\frac{\epsilon}{2}.
\end{aligned}
\end{equation*}
This is a contradiction, since $u$ satisfies the a priori estimates (\ref{eq:88}) and  (\ref{eq:106}) and the estimate (\ref{873}). Hence $\tau(u)=1$ and consequently $u=\tilde{T}(u)=\tau(u)T(u)=T(u)$. Thus $u\in D_0^{1,2}(\R^n)\cap C_b^2(\R^n)$ is a fixed point of $T$ and, equivalently, solves (\ref{eq:mainequation}).
\end{proof}

Finally, Theorem \ref{th:main} follows from Theorem \ref{th:smoothexistence} by a simple approximation argument.

\begin{proof}[Proof of Theorem \ref{th:main}]
By standard approximation argument there exists a sequence $(\rho_k)_{k=1}^\infty\subset C_c^\infty(\R^n)$ such that
\begin{equation}\label{eq:96}
\Vert \rho_k-\rho\Vert_{\mathcal{X}^*}+\Vert \rho_k-\rho\Vert_{L^p(\R^n)}\xrightarrow{k\rightarrow \infty} 0.
\end{equation}
By Theorem \ref{th:smoothexistence} there exists a unique weak solution $u_k\in C^{1,\alpha}(\R^n)\cap D_0^{1,2}(\R^n)$ to the equation
\begin{equation}\label{eq:97}
-\div\Big(\frac{\D u_k}{\sqrt{1-|\D u_k|^2}}\Big)=\rho_k.
\end{equation}
Moreover, there exists a constant $\theta=\theta(n,p,\Vert \rho\Vert_{\mathcal{X}^*},\Vert\rho\Vert_p)>0$ such that
\begin{equation}\label{eq:98}
\Vert\D u_k\Vert_\infty\leq 1-\theta,
\end{equation}
for all $k\in\N$ large enough. We also have
\begin{equation}\label{eq:99}
\Vert u_k\Vert_{C^{1,\alpha}(\R^n)}\leq K,
\end{equation}
where $\alpha=\alpha(n,p,\Vert \rho\Vert_{\mathcal{X}^*},\Vert\rho\Vert_p)\in(0,1]$ and $K=K(n,p,\Vert \rho\Vert_\mathcal{X^*},\Vert\rho\Vert_p)>0$, when $k\in\N$ is large enough. By (\ref{eq:99}), the sequence $(u_k)_{k=1}^\infty$ is a sequence of uniformly bounded and equicontinuous functions. By Arzela-Ascoli theorem we may assume that, up to a subsequence, the sequence $(u_k)_{k=1}^\infty$ converges locally uniformly to $u\in C^{1,\alpha}(\R^n)$ and, similarly,  $(\D u_k)_{k=1}^\infty$ converges locally uniformly to $\D u$. By locally uniform convergence, the limit function $u$ also satisfies estimates (\ref{eq:98}) and (\ref{eq:99}). Moreover, by Lemma \ref{lm:derivativeL2estimate}, the sequence $(\D u_k)_{k=1}^\infty$ is bounded in $L^2(\R^n)$ and hence, up to a subsequence, converges weakly in $L^2(\R^n)$ to $\D u$. Thus we have $u\in D_0^{1,2}(\R^n)$ and
\begin{equation*}
\Vert u\Vert_{\D_0^{1,2}(\R^n)}\leq \Vert \rho\Vert_{\mathcal{X}^*}.
\end{equation*}
By (\ref{eq:97}), (\ref{eq:98}), the locally uniform convergence of $(D u_k)_{k=1}^\infty$, and (\ref{eq:96}), we have
\begin{equation*}
\begin{aligned}
\int_{\R^n}\langle \D F(\D u),\D \varphi\rangle\dx &=\lim_{k\rightarrow\infty}\int_{\R^n}\langle \D F(\D u_k),\D \varphi\rangle\dx \\
&=\lim_{k\rightarrow\infty}\int_{\R^n} \rho_k \varphi\dx  \\
&=\int_{\R^n} \rho \varphi\dx,
\end{aligned}
\end{equation*}
for any $\varphi\in C_c^\infty(\R^n)$. Thus $u$ is the unique weak solution to (\ref{eq:mainequation}).
\end{proof}



\begin{thebibliography}{9}     
\bibitem{BS}\textsc{Bartnik, R. and Simon, L.}, \emph{Spacelike hypersurfaces with prescribed boundary values and mean curvature.} Commun. Math. Phys. 
\textbf{87} (1982), 131--152.

\bibitem{BdAP}\textsc{Bonheure, D., d'Avenia, P. and Pomponio A.}, \emph{On the electrostatic Born--Infeld equation with extended charges} Commun. Math. Phys. 
\textbf{346} (2016), 877--906.

\bibitem{BI}\textsc{Bonheure, D., Iacopetti, A.}, \emph{On the regularity of the minimizer of the electrostatic Born--Infeld energy}, Arch. Ration. Mech. Anal. \textbf{232} (2019), 697--725.

\bibitem{B1}\textsc{Born, M.}, \emph{Modified field equations with a finite radius of the electron.} Nature, \textbf{132} (1933), 282.

\bibitem{B2}\textsc{Born, M.}, \emph{On the quantum theory of the electromagnetic field.} Proc. R. Soc. Lond. A, \textbf{143} (1934), 410--437.

\bibitem{BI1}\textsc{Born, M. and Infeld, L.}, \emph{Foundations of the new field theory.} Nature, \textbf{132} (1933), 1004.

\bibitem{BI2}\textsc{Born, M. and Infeld, L.}, \emph{Foundations of the new field theory.} Proc. R. Soc. Lond. A, \textbf{144} (1934), 425--451.

\bibitem{C}\textsc{Calabi, E.}, \emph{Examples of Bernstein problems for some nonlinear equations.} In: \emph{Global Analysis.} Providence, RI: American Mathematical Society, 1970.

\bibitem{CY}\textsc{Cheng, S.-Y. and Yau S.-T.}, \emph{Maximal space-like hypersurfaces in the Lorentz-Minkowski spaces.} Annals of Mathematics, \textbf{104} (1976), 407--419.

\bibitem{GT}\textsc{Gilbarg, D. and Trudinger, N.}, \emph{Elliptic partial differential equations of second order.} Second Edition, New York: Springer, 2001.




\end{thebibliography}
\end{document}